\documentclass[11pt]{amsart}

\usepackage{epsf,amssymb,amsmath,overpic,amscd,psfrag,stmaryrd,times, mycommands, mycommands2, mathrsfs, euscript,thmtools, amsbsy, mathtools}

\usepackage[style=alphabetic,backend=biber]{biblatex}

\usepackage[all]{xy}
\usepackage{hyperref, stmaryrd, adjustbox}
\usepackage[dvipsnames]{xcolor}
\usepackage{enumitem}
\usepackage{tikz,tikz-cd}
\usepackage{todonotes}

\usepackage{caption}
\usepackage{subcaption}

\input{biblatex.cfg}
\makeatletter
\newcommand{\labitem}[2]{%
\def\@itemlabel{#1}
\item
\def\@currentlabel{#1}\label{#2}}
\makeatother

\begin{document}
\title[Morse homology and equivariance]
{Morse homology and equivariance} 

\author{Erkao Bao}
\address{School of Mathematics, University of Minnesota, Minneapolis, MN 55455}
\email{bao@umn.edu} 
\urladdr{https://erkaobao.github.io/math/}

\author{Tyler Lawson}
\address{School of Mathematics, University of Minnesota, Minneapolis, MN 55455}
\email{tlawson@math.umn.edu} 
\urladdr{https://www-users.cse.umn.edu/~tlawson/}




\thanks{Erkao Bao is supported by NSF Grants DMS-2404529.}
\thanks{Tyler Lawson is supported by NSF Grants DMS-2208062.}

\begin{abstract}
In this paper, we develop methods for calculating equivariant homology from equivariant Morse functions on a closed manifold with the action of a finite group. We show how to alter a $G$-equivariant Morse function to a \emph{stable} one, where the descending manifold from a critical point $p$ has the same stabilizer group as $p$, giving a better-behaved cell structure on $M$. For an equivariant, stable Morse function, we show that a generic equivariant metric satisfies the Morse--Smale condition.

In the process, we give a proof that a generic equivariant function is Morse, and that equivariant, stable Morse functions form a dense subset in the $C^0$-topology within the space of all equivariant functions.

Finally, we give an expository account of equivariant homology and cohomology theories, as well as their interaction with Morse theory. We show that any equivariant Morse function gives a filtration of $M$ that induces a Morse spectral sequence, computing the equivariant homology of $M$ from information about how the stabilizer group of a critical point acts on its tangent space. In the case of a stable Morse function, we show that this can be further reduced to a Thom-Smale-Witten complex.
\end{abstract}

\maketitle


\section{Introduction}

Morse homology has been extensively studied, offering profound insights into manifold topology. However, despite significant progress, such as \cite{wasserman1969equivariant, hingston1981equivariant, austin1995morse,cho2014orbifold}, equivariant Morse homology remains relatively underdeveloped, primarily due to challenges in integrating symmetry and genericity.

Here are some of the challenges associated with the study of equivariant Morse functions on compact manifolds $M$:
\begin{itemize}
    \item Given an equivariant Morse function, the existence of an equivariant Riemannian metric ensuring the Morse--Smale property—where ascending and descending manifolds intersect transversely—is not guaranteed and is often obstructed.
    \item Classical Morse functions on $M$ are closely related to \emph{cell decompositions} of $M$. A gradient flow gives each critical point $p$ of index $k$ a descending manifold diffeomorphic to the open disk $D^k$. In the equivariant case, however, this disk inherits a potentially nontrivial linear action of the stabilizer group $\stab(p)$. This means that in the equivariant case, Morse functions are related to decompositions into \emph{representation cells}, rather than the ordinary cells which appear in the definition of a $G$-CW complex. Here, we recall that a $G$-CW complex is a CW complex on which $G$ acts by cellular maps such that the fixed point set of any element of $G$ is a subcomplex. In particular, it implies that inside the interior of each cell, the stabilizer is constant.
    \item As consequence of this difficulty, if a critical point $p$ is fixed by a subgroup $H \subset G$, then the Morse index $\morseIndex{f}(p)$ in $M$ may differ from the index $\morseIndex{f|_{M^H}}(p)$ in the fixed-point manifold $M^H$. For example, we may have a critical point $p$ which is of index $5$ in $M$ and of index $2$ in $M^G$, while $q$ has index $4$ in $M$ and index $3$ in $M^G$. This makes it impossible to give a reasonable definition of a ``self-indexing equivariant Morse function'' which is generic among Morse functions.
    \item There is a much wider variety of homology and cohomology theories in the equivariant case, including Borel-equivariant and Bredon theories, with more subtle interactions.
\end{itemize}

In this paper, we propose an approach to address these challenges. By equivariantly perturbing the Morse function, we induce a \emph{stability} property, which forces the descending manifold of a critical point $p$ to be pointwise fixed by the stabilizer group of $p$. This makes the critical points correspond to cells in a $G$-CW structure.

Under this stability condition, we also demonstrate that a generic equivariant Riemannian metric renders the function-metric pair Morse--Smale, enabling the calculation of equivariant Morse homology. Furthermore, the stability condition also implies that the Morse chain complex of the fixed point set $M^H$ becomes a subcomplex of the Morse chain complex of $M$.
This approach also gives guidance on defining equivariant Morse homology for the non-Morse--Smale case without perturbing the Morse function, a special case of which is Kronheimer-Mrowka's definition of Morse homology for manifolds with boundaries, treated as closed manifolds with a $C_2$-action (see \cite{Kronheimer2007Monopoles}).

In the remainder of the paper, we give a brief introduction to equivariant homology and cohomology theories (both ``ordinary'' and ``generalized''), and how (co)homology interacts with equivariant Morse theory. Examples of ordinary equivariant cohomology theories (also known as Bredon cohomology), which satisfy an analogue of the dimension axiom, include:
\begin{align*}
    X &\mapsto H^*(X),\\
    X &\mapsto H^*(X^G),\\
    X &\mapsto H^*(X/G).
\end{align*}
Examples of generalized ones include the Borel and Tate cohomology theories. 
For any such equivariant (co)homology theory, we discuss how an equivariant Morse function makes $M$ homotopy equivalent to a $G$-space built by attaching the unit discs in representations of $G$. The increasing filtration of $M$ by level sets produces a Morse spectral sequence for calculating the equivariant (co)homology of $M$. The starting components of this spectral sequence are the equivariant (co)homology groups for these representation cells associated to the critical points of $M$.

In the case of a stable Morse function and an ordinary equivariant (co)homology theory, we show that we can reduce further. The (co)homology of $M$ is directly calculated by a \emph{Thom-Smale-Witten complex} involving the critical points of $M$ and their stabilizer groups.

Finally, the techniques developed in this paper can also be used to define orbifold (co)homology over integers \cite{cho2014orbifold}.

\subsection{Statements and outline}

Let $G$ be a finite group, and $M$ a smooth manifold of dimension $n$ with a smooth action of the group $G$. The vector space $C^\infty(M)$, of smooth functions $f\co M \to \R$, has an action of the group $G$ via $f^{\sigma} = f \circ \sigma$. The fixed subspace $C^k(M)^G$ is the space of \emph{equivariant} real-valued functions.

We denote the critical point set of $f$ by $\crit(f) = \{ p \mid df_p = 0 \}$, and for $p \in \crit(f)$ the {\em Hessian} of $f$ at $p$ by $\hess{f}(p)\co T_p M \times T_p M \to \R$.


\begin{definition}[Morse function]
    An \emph{equivariant Morse function} on $M$ is an element $f \in C^\infty(M)^G$ whose underlying map $f\co M \to \R$ is Morse: for all $p \in \crit(f)$, the Hessian $\hess{f}(p)$ is nondegenerate.
\end{definition}

We include a short alternative proof of the result of \cite[4.10]{wasserman1969equivariant} that equivariant Morse functions exist and are generic \footnote{After we posted the first version of this paper, it was pointed out that a part of our alternative proof is similar to the treatment in Section 3.4 of \cite{bai2020equivariant}.}. Many of the standard proofs in the nonequivariant setting rely on a transversality result which does not apply equivariantly. 

\begin{theorem}
    Equivariant Morse functions are generic in the $C^k$-topology for any $k\geq 2$. 
\end{theorem}

To study Morse theory, we need an equivariant metric $g$ on $M$.
For a critical point $p\in \crit(f)$, let $\stableManifold{f}{\g}{p}$ denote the ascending manifold of $p$, and $\unstableManifold{f}{\g}{p}$ denote the descending manifold of $p$.

\begin{definition}[Morse--Smale condition]
    The pair $(f, \g)$ is defined to be {\em Morse--Smale} if, for all $p, q \in \crit (f)$, the ascending manifold $\unstableManifold{f}{\g}{p}$ of $p$ intersects the descending manifold $\stableManifold{f}{\g}{q}$ of $q$ transversely.
\end{definition}

For any Morse function $f$, it is well-known that a generic metric $\g$ makes the pair $(f,\g)$ Morse--Smale. However, this generalization does not hold in the presence of a group action. Indeed, for some equivariant Morse functions, there does not exist any equivariant metric that satisfies the Morse--Smale condition. To handle this problem, we perturb the Morse function.

For any $p \in M$ with stabilizer $\stab(p) = H$, the tangent space $T_p M$ has an induced action of $H$. It canonically decomposes as the direct sum of the fixed subspace $(T_p M)^H = T_p(M^H)$ and a complement $T'_p M$, the kernel of the averaging operator, i.e.,
\[
T'_p M = \{\vec w \in T_p M \mid \sum_{\sigma \in \stab(p)} d\sigma_p \vec w = 0\}.
\]
For any equivariant metric on $M$, $T'_p(M)$ is the orthogonal complement of the fixed space $T_p(M)^H$.

\begin{definition}
    A critical point $p$ of an equivariant Morse function $f$ is \emph{stable} if the Hessian $\hess{f}(p)$ is positive definite on $T'_p M$.
\end{definition}

\begin{definition}
    An equivariant Morse function $f$ is {\em stable} if all of its critical points are stable.
\end{definition}

\begin{remark}
    If $M$ has an equivariant metric, then we can take orthogonal complements. In this case, stability at a critical point $p$ with $\stab(p) = H$ asks that the negative eigenspace of the Hessian be contained inside the fixed subspace $T_p(M)^H$. Equivalently, the descending manifold of the Morse flow from $p$ is also $H$-fixed.

    In terms of the CW-decomposition of $M$ determined by the pair $(f,g)$, an equivariant Morse function usually decomposes $M$ into cells of the form $G \times_H V$ where $V$ is a finite-dimensional representation of $H$. Stability is the requirement that $M$ decomposes into cells of the form $G/H \times \R^n$.
\end{remark}

We give an alternative proof of the following result of \cite{mayer1989Ginvariant}\footnote{We thank Bernhard Hanke for pointing out that this result is already known in \cite{mayer1989Ginvariant}.}:
\begin{theorem}
\label{thm:stableperturbation}
    Given an equivariant Morse function $f$ on $M$, there exists an equivariant {\em stable} Morse function $f'$ that is arbitrarily close to $f$ in the $C^0$-topology and that agrees with $f$ except near the critical points of $f$.
\end{theorem}
We show:
\begin{theorem}[Equivariant Smale Theorem] \label{thm: morse smale is generic fixing a stable f}
    Let $k\geq 1$ be an integer. Let $f$ be an equivariant, stable Morse function of class $C^{k+1}$. Then, for a generic equivariant metric $g$ of class $C^k$, the pair $(f,g)$ is Morse--Smale.
\end{theorem}

With these results, we can define the equivariant Morse chain complex with the fixed point set being a sub-complex, and one implication is a Morse theoretical proof of the Smith inequality, which was originally proved in \cite{floyd1952on}: 
\begin{corollary}[Smith inequality]
  Let $p$ be a prime number, and let $M$ be a closed manifold with a finite $p$-group $G$ acting smoothly. Then the following hold:
  \begin{enumerate}
  \item[(i)] For all $\ell \geq 0$, we have
  \[
      \sum_{k=\ell}^{\infty} \dim H_k(M^G; \mathbb{F}_p) \leq \sum_{k=\ell}^{\infty} \dim H_k(M; \mathbb{F}_p).
  \]
  \item[(ii)] Let $\chi_p$ be the mod $p$ Euler characteristic. Then $\chi_p(M^G) \equiv \chi_p(M)$.
  \end{enumerate}
\end{corollary}
Starting with an arbitrary equivariant Morse function, we perturb it to a stable one. Then, by Theorem~\ref{thm: morse smale is generic fixing a stable f}, we can construct the equivariant Thom-Smale-Witten complex. For computational purpose and for adapting the theory to Floer homology, we do not want to perturb the Morse function.
By reducing the perturbation to zero, we recover the Thom-Smale-Witten complex using the unperturbed moduli spaces. But instead of counting smooth Morse trajectories, one also needs to include certain broken trajectories. This process is explained in detail in \cite{bao2024equivariantMorseHomologyFor}.

\section{Morse polynomials}

Suppose that $V$ is a real inner product space of dimension $n$. We write $P(V)$ for the polynomial algebra on $V$, and $P_d(V)$ for the subspace of polynomials of degree at most $d$.

For any point $p \in V$, evaluation at $p$ determines a quotient map $P(V) \to \R$ whose kernel is the associated maximal ideal $m_p$. More generally, for any $p$ and any $k \geq 0$, there is a \emph{Taylor expansion} map
\[
    P(V) \to P(V) / m_p^k \cong P_k(V)
\]
sending a function $f$ to its degree-$k$ Taylor expansion about $p$.

\begin{proposition}
    Suppose that $V$ is a real vector space of dimension $n$, that $p_1,\dots,p_d$ are distinct points in $V$, and that $k \geq 0$. Then there exists an $N$ such that the map
    \[
    P_N(V) \to \prod_i P(V) / m_{p_i}^k
    \]
    is onto. Moreover, $N$ depends only on $d$ and $k$.
\end{proposition}

The reader will probably recognize this result and its proof as the Chinese remainder theorem, except we need to track an effective bound on $N$.

\begin{proof}
Let
\[
\phi_j(x) = \prod_{i \neq j} \frac{||x-p_i||^2}{ ||p_j-p_i||^2}.
\]
By construction, $\phi_j$ is a polynomial of degree $2d-2$ such that $\phi_j(p_j)=1$ and $\phi_j(p_i) = 0$ for $j \neq i$; equivalently, $\phi_j \equiv 1 \mod m_j$ and $\phi_j \equiv 0 \mod m_i$ for $i \neq j$.

The polynomial $1 - (1 - \phi_j^k)^k$ of degree $(2d-2)k^2$ is then congruent to $1$ mod $m_j^k$ and $0$ mod $m_i^k$ for $j \neq i$. As a result, for any desired Taylor polynomials $f_1, \dots, f_d$ of degree less than $k$, the polynomial
\[
f = \sum_j f_j \cdot (1 - (1 - \phi_j^k)^k)
\]
is a polynomial of degree less than $(2dk-2k-1)k$ with the desired Taylor expansions.
\end{proof}

\begin{corollary}
\label{cor:equivariantlifting}
    Suppose that $V$ has an action of a group $G$ and that $k \geq 0$. Then there exists an $N$ such that for all $p \in V$ with stabilizer group $H \subset G$, the map
    \[
    P_N(V)^G \to \left(P(V)/m_p^k\right)^H,
    \]
    given by Taylor expansion at $p$, is onto. In particular, a Taylor expansion at $p$ lifts to an equivariant polynomial (of degree at most $N$) if and only if it is $H$-fixed.
\end{corollary}

\begin{proof}
   For any $p$ in $V$ with stabilizer $H$, the orbit $G \cdot p$ is in bijection with the set $G/H$ and hence has size smaller than $|G|$. The Taylor expansion map
   \[
   P_N(V) \to \prod_{[g] \in G/H} P(V)/m_{g\cdot p}^k
   \]
   is a $G$-equivariant surjection for some $N$ depending only on $|G|$ and $k$. Taking $G$-fixed points, we get a surjection
   \[
   P_N(V)^G \to \left(\prod_{[g] \in G/H} P(V)/m_{gp}^k\right)^G \cong \left(P(V)/m_{p}^k\right)^H
   \]
   because taking fixed points is exact in characteristic zero.
\end{proof}

\section{Local Morse functions}

\begin{proposition}
\label{prop:localmorsegeneric}
    Suppose $V$ is a $G$-inner product space of dimension $n$. For sufficiently large $N$, Morse functions are generic in $P_N(V)^G$ .
\end{proposition}    

\begin{proof}
    Fix $N$ as in Corollary~\ref{cor:equivariantlifting} with $k=2$, so $stab(p)$-equivariant quadratic Taylor expansions at $p$ all lift to $G$-equivariant of degree at most $N$. Let $d = \dim(P_N)$. A polynomial function $f \in P_N$ is specifically equivalent data to the coefficients in its Taylor expansion at the origin, so $d = \binom{n+N}{n}$.

    Let
    \[
    U(H) = \{p \in V \mid stab(p) = H\},
    \]
    which is a Zariski open set $V^H \setminus \cup_{H \subset K} V^K$. There is a map of real algebraic varieties
    \[
    V^H \times P_N(V)^G \to V^H \times P_2(V)^H
    \]
    given by $(p,f) \mapsto (p,q)$ where $q(x-p)$ is the quadratic Taylor approximation of $f$ at $p$. This restricts to a surjective map
    \[
    T\co U(H) \times P_N(V)^G \to U(H) \times P_2(V)^H
    \]
    by Corollary~\ref{cor:equivariantlifting}.

    The subset
    \[
    N(H) = \{(p,f) \in U(H) \times (P_N)^G | f\text{ is not Morse at }p\}
    \]
    is a closed subset of codimension $(\dim(V^H)+1)$, as follows. If $T(p,f) =(p,q)$ in $U(H) \times P_2(V)^H$, then $f$ is not Morse at $p$ if and only if $q$ is not Morse at the origin; by surjectivity of $T$ it suffices to show that non-Morse functions are of codimension $(\dim(V^H) + 1)$ in $P_2(V)^H$.
    
    Note that $P_2(V)^H = \R \oplus V^H \oplus Sym^2(V)^H$, decomposing a polynomial into its constant, linear, and quadratic terms. The polynomials which are critical at the origin are those with vanishing linear term---the component in $V^H$---and so form a linear subspace of $P_2(V)^H$ of codimension $\dim(V^H)$. Degeneracy of this critical point is equivalent to degeneracy of the Hessian---vanishing of the determinant of the matrix of second partial derivatives---a degree-$n$ polynomial in the coefficients of the quadratic term. However, $||x||^2$ is a degree-2 $H$-equivariant polynomial with a nondegenerate Hessian, so this Hessian determinant is not identically the zero function on $Sym^2(V)^H$. Because the Hessian determinant is not the zero polynomial, the non-Morse functions are of codimension $(\dim(V^H)+1)$.

    Therefore, the dimension of $N(H)$ is
    \[
    \dim(U(H)) + \dim(P_N(V)^G) - \dim(V^H) - 1 = \dim(P_N(V)^G) - 1.
    \]
    The composite projection $N(H) \to P_N(V)^G$ has an image which is therefore of measure zero. As $H$ varies among subgroups of $G$, the union of these images is therefore nowhere dense. However, the union of these images consists precisely of non-Morse functions, as desired.
\end{proof}

\begin{theorem}
    Suppose $V$ is a $G$-inner product space of dimension $n$ and $2 \leq k$. Then for any $G$-equivariant compact subset $K \subset V$, Morse functions are open and dense in $C^k(K)^G$.
\end{theorem}

\begin{proof}
    Openness is already shown in \cite[Lemma B]{milnor1965lectures}, and so it suffices to prove density.
    
    Let $f \in C^k(K)^G$. The polynomial functions $P(V)$ are dense in $C^k(K)$, and by the averaging trick we find that equivariant polynomials $P(V)^G$ are dense in $C^k(K)^G$. Proposition~\ref{prop:localmorsegeneric} then shows that equivariant Morse polynomials are dense in $C^k(K)^G$. Putting these together, equivariant Morse functions are dense in $C^k(K)^G$.
\end{proof}

\section{Equivariant tubular neighborhoods}
\begin{lemma}
    Any smooth $G$-equivariant manifold $M$ has a smooth $G$-equivariant metric.
\end{lemma}

\begin{proof}
    If we choose any smooth metric $g$ on $M$, then $\tfrac{1}{|G|} \sum_{\sigma \in G} g^\sigma$ is a $G$-equivariant metric.
\end{proof}

\begin{proposition}
    Suppose $N \subset M$ is a closed submanifold with normal bundle $\nu$, and $H \subset G$ is a subgroup preserving $N$. Then there exists an $H$-equivariant tubular neighborhood of $N$: an open neighborhood $N \subset U \subset M$ preserved by $H$ and an $H$-equivariant diffeomorphism $\nu \to U$ which is the identity on the zero-section. 
\end{proposition}

\begin{proof}
    Choose a $G$-equivariant metric on $M$. Then geodesic flow defines an exponential map $\exp\co V \to M$ for some open neighborhood $N \subset V \subset \nu$, and $d\exp_(p,0)$ is the canonical isomorphism for any $p \in N$. By the inverse function theorem and compactness of $N$, there exists an $\epsilon > 0$ such that $\exp$ induces a diffeomorphism $\{\vec w \in \nu \mid ||\vec w|| < \epsilon\} \to U$, sending for some open neighborhood of $N$.

    Moreover, for any $\sigma \in H$, $\exp \circ \sigma = \sigma \circ \exp$, because both sides are defined by geodesic flows; hence $U$ is preserved by $H$ and $\exp$ is $H$ equivariant.
        
    The function
    \[
    \vec w \mapsto \frac{\epsilon \vec w }{\epsilon^2 - ||\vec w||^2}
    \]
    is an equivariant diffeomorphism $\{\vec w \in \nu \mid ||\vec w|| < \epsilon\} \to \nu$, and precomposing with its inverse gives the desired diffeomorphism $\nu \to U$.
\end{proof}

\begin{corollary}
\label{cor:tangentchart}
    For any $p \in M$ with stabilizer $\stab(p) = H$, there exists an open neighborhood near $p$ which is $H$-equivariantly diffeomorphic to the tangent space $T_p M$ with its induced $H$-action.
\end{corollary}

\section{Global Morse functions}

Now that we know that equivariant Morse functions are locally generic, we can proceed as Milnor did to show that the same is true globally.

\begin{proposition}
    For a smooth closed $G$-manifold $M$  and $k \geq 2$, Morse functions are generic on $C^\infty(M)^G$ in the $C^k$-topology.
\end{proposition}

\begin{proof}
    Each $p \in M$ with stabilizer $H$ has an $H$-equivariant open neighborhood $U$ diffeomorphic to open discs in $V$ for some $H$-inner product space $V$ of $H$, by Corollary~\ref{cor:tangentchart}. We may choose $V$ sufficiently small that $V \cap {}^\sigma V = \emptyset$ for $\sigma \not\in H$.
    
    We proceed similarly to the proof of \cite[Theorem 2.7]{milnor1965lectures}. Find a finite cover $U_i$ of such $H_i$-equivariant coordinate charts with closed discs $D_i$ around the origin that cover $M$.

    Recall from \cite[Lemma B]{milnor1965lectures} that, for a compact subset $K \subset M$, being Morse on $K$ is an open condition in the $C^k$-topology. Suppose $f$ is in $C^\infty(M)^G$. By induction, assume that we can perturb $f$ an arbitrarily small amount to be Morse on $D_1 \cup \dots \cup D_{i-1}$. Any sufficiently small perturbations of $f$ in the $C^k$-topology also remain Morse on $D_1 \cup \dots \cup D_{i-1}$, and so it suffices to show that there is always a small perturbation of $f$ which is Morse on $D_i$.
    
    Using a cutoff function, we can then approximate $f$ arbitrarily closely on $U_i$ by an $H_i$-equivariant smooth function $h$ which is Morse on $D_i$ and which agrees with $f$ near the boundary of $U_i$. Since the orbits ${}^\sigma U_i$ do not intersect for distinct $[\sigma] \in G/H_i$, setting $h(\sigma x) = h(x)$ gives a well-defined extension to a $G$-equivariant smooth function on $W_i = \cup_{\sigma \in G/H_i} {}^\sigma U_i$, Morse on $D_i$, which agrees with $f$ near the boundary of $W_i$. We can then extend this smooth function to $M$ by setting it equal to $f$ outside $W_i$, obtaining an arbitrarily small perturbation of $f$ which is $G$-equivariant and Morse on $D_i$.
\end{proof}

\section{Stable Morse functions}

Our goal in this section is to prove Theorem~\ref{thm:stableperturbation}:  any equivariant Morse function on $M$ with a critical point $p$ can be altered in an arbitrarily small neighborhood $U$ of $p$ to create a new Morse function, having only stable critical points in $U$.

\begin{remark}
    The altered Morse function will still have a critical point at $p$ of a different index. As a result, it involves altering the signature of the Hessian and cannot be done with a small perturbation of the original function in the $C^2$-topology.
\end{remark}

We recall the following equivariant Morse lemma from \cite[Lemma 4.1]{wasserman1969equivariant}.

\begin{proposition}
    Suppose the critical point $p$ has $\stab(p) = H$. Then there exists an $H$-equivariant coordinate neighborhood of $p$ diffeomorphic to an open disc in $V \oplus W$ for some $H$-inner product spaces $V$ and $W$, such that in these coordinates the Morse function is given by $f(v,w) = ||v||^2 - ||w||^2$.
\end{proposition}

We will need a cutoff-type function satisfying certain constraints.
\begin{proposition}
    There exists a smooth function $\phi\co \R \to \R$ satisfying the following properties:
    \begin{itemize}
    \item It is odd: $\phi(-t) = -\phi(t)$.
    \item It satisfies $\phi(t) = 1$ for $t \geq 1$.
    \item It is nondecreasing: $\phi'(t) \geq 0$ for all $t$.
    \item The second derivative $\phi''(t)$ has a unique zero in $(-1,1)$. (In particular, it is positive on $(-1,0)$.)
    \end{itemize}
\end{proposition}

From this, we can deduce the following:
\begin{corollary}
    The function $-t^2 \phi(t-2)$ has exactly two critical points $0$ and $t_0$, satisfying $1 < t_0 < 2$. The second derivative at $t_0$ is negative.
\end{corollary}

\begin{proof}
    The first derivative is $-2t \phi(t-2) - t^2 \phi'(t-2)$, and so any nonzero critical point is a zero of $2\phi(t-2) + t \phi'(t-2)$. This function is strictly negative for $t \leq 1$ and strictly positive for $t > 2$, so there exists at least one zero in the desired range. Moreover, the second derivative is $3 \phi'(t-2) + t \phi''(t-2)$, which is positive for $1 < t < 2$, and so this zero is unique.
\end{proof}

We also require a smooth plateau function $\psi\co \R \to \R$ which is constant with value $1$ in an interval $[t_0-\delta,t_0 + \delta]$ and constant with value $0$ outside an interval $[1+\delta,3-\delta]$.

\begin{construction}
\label{con:criticalperturbation}
    Given any $H$-inner product spaces $V$, $W$, and $U$, together with a smooth $H$-equivariant function $h$ from the unit sphere of $U$ to $\R$, we define
    \[
    f(v,w,u) = ||v||^2 - ||w||^2 - ||u||^2\phi(||u||-2) + \epsilon \psi(||u||) h(\tfrac{u}{||u||}),
    \]
    where $\epsilon > 0$ is a fixed constant.
\end{construction}

\begin{proposition}
    The function $f\co V \oplus W \oplus U \to \R$ of Construction~\ref{con:criticalperturbation} satisfies the following properties.
    \begin{enumerate}
    \item For $||u|| \geq 3$, it agrees with the function $||v||^2 - ||w||^2 - ||u||^2$. In particular, it has no critical points in this range.
    \item For $||u|| \leq 1$, it agrees with the function $||v||^2 - ||w||^2 + ||u||^2$. In particular, the origin is the only critical point in this range and the Hessian is positive definite on $V \oplus U$.
    \item The function $f$ is smooth and equivariant.
    \item For sufficiently small $\epsilon > 0$, the other critical points are precisely points of the form $(0,0,t_0 u)$ for $u$ a point on the unit sphere which is a critical point of $h$.
    \item At such a critical point $t_0 u$, the negative eigenspace of the Hessian is contained in $W \oplus U$.
    \item If $u$ is a nondegenerate critical point of $h$, then $(0,0,t_0 u)$ is a nondegenerate critical point of $f$.
    \item If $u$ is a stable critical point of $h$ and $W$ is fixed by $H$, then $(0,0,t_0 u)$ is a stable critical point of $f$.
    \end{enumerate}
\end{proposition}

\begin{proof}
    \begin{enumerate}
    \item In this range, $\phi(||u||) = 1$ and $\psi(||u||) = 0$.
    \item Similarly, in this range, $\phi(||u||) = -1$ and $\psi(||u||) = 0$.
    \item Equivariance is clear from equivariance of the absolute value and the function $h$. Smoothness is also clear away from $||u|| = 0$; near $||u||=0$, however, smoothness is clear from the previous point.
    \item Away from $||u||=0$, we can express points in the form $(v, w, t u(y))$ for $t > 0$ and $y$ local coordinates on the unit sphere of $U$. In these coordinates, the gradient of $f$ takes the form
    \[
    \nabla f = (2v, -2w, \tfrac{\partial f}{\partial t}, \epsilon \psi(t) \nabla h).
    \]
    Therefore, to have a critical point we need $v = w = 0$; we also need either $\psi(t) = 0$ or $u$ is a critical point of $h$. For such points, the function reduces to
    \[
    f(0,0,tu) = -t^2 \phi(t-2) + \epsilon \psi(t) h(u).
    \]
    It remains to determine when the partial derivative with respect to $t$,
    \[
    \tfrac{\partial f}{\partial t}(0,0,tu) = -2t \phi(t-2) - t^2 \phi'(t-2) + \epsilon \psi'(t) h(u),
    \]
    vanishes.

    The plateau function $\psi(t)$ is constant except on intervals $[1+\delta, t_0 - \delta]$ and $[t_0 + \delta, 3 - \delta]$. On these intervals, $-2t \phi(t-2) - t^2 \phi'(t-2)$ is nowhere zero and both $\psi'(t)$ and $h$ are bounded. Therefore, there exists an $\epsilon$ sufficiently small that
    \[
    |\epsilon \psi'(t) h(u)| < |-2t \phi(t-2) - t^2 \phi'(t-2)|,
    \]
    and hence $\tfrac{\partial f}{\partial t}$ does not vanish.

    Finally, for $t$ outside these intervals, the function $\psi(t)$ is constant and so $\psi'(t) = 0$. Therefore, in the remaining regions,
    \[
    \tfrac{\partial f}{\partial t}(0,0,tu) = -2t \phi(t-2) - t^2 \phi'(t-2)
    \]
    which vanishes precisely at $t = t_0$.
    \item Again in the local coordinate system $(v, w, t u(y))$, the Hessian of $f$ is given in block form:
    \[
    \hess{f}(p) =
    \begin{bmatrix}
        2 Id & 0 & 0 & 0\\
        0 & -2 Id & 0 & 0\\
        0 & 0 & \psi(t) \hess{h}(u) & \epsilon \psi'(t) \nabla h\\
        0 & 0 & \epsilon \psi'(t) \nabla h^T & \tfrac{\partial^2 f}{\partial t}^2\\
    \end{bmatrix}
    \]
    At a critical point, $t = t_0$ and $\nabla h = 0$, and the second derivative $\tfrac{\partial^2 f}{\partial t}^2 (-t^2 \phi(t))$ is a negative constant $-c < 0$, so this reduces to
    \[
    \begin{bmatrix}
        2 Id & 0 & 0 & 0\\
        0 & -2 Id & 0 & 0\\
        0 & 0 & \hess{h}(u) & 0\\
        0 & 0 & 0 & -c
    \end{bmatrix}
    \]
    In particular, the negative eigenspace is contained inside $W \oplus U$.
    \item The block form of the Hessian at $(0,0,t_0 u)$ also makes it clear that the Hessian is nondegenerate at $(0,0,t_0 u)$ if and only if the Hessian of $h$ is nondenegerate at $u$.
    \item By the block form, the negative eigenspace of the Hessian at $(0,0,t_0 u)$ is spanned by $W$, by the negative eigenspace of $\hess{h}(u)$, and by $u$. Let $K$ be the stabilizer of this critical point $(0,0,t_0 u)$; $K$ is clearly $\stab(u) \subset H$. Then $u$ is fixed by $K$ by assumption. If $W$ is $H$-fixed then it is also $K$-fixed. Finally, if the critical points of $h$ are stable, then by assumption the negative eigenspace of $\hess{h}$ is also $K$-fixed. Therefore, if both of these hold, $(0,0,t_0 u)$ is a stable critical point.
    \end{enumerate}
\end{proof}

\begin{figure}
\centering
\begin{subfigure}[t]{.45\textwidth}
  \centering
  \includegraphics[scale = 0.25]{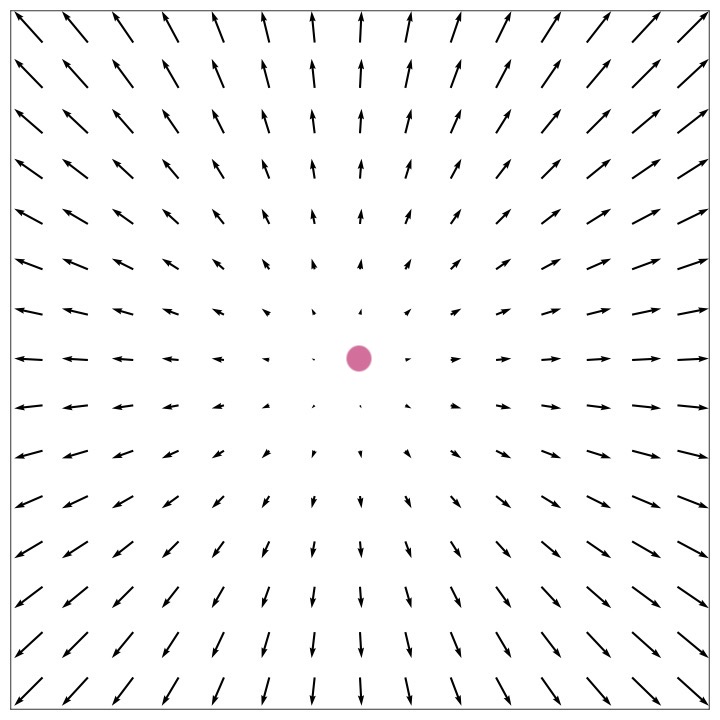} 
  \caption{\footnotesize Before modification, $o$ has index $2$ in $X$, and index $0$ in $X^G$.}
\end{subfigure}\quad\quad
\begin{subfigure}[t]{.45\textwidth}
  \centering
  \includegraphics[scale = 0.25]{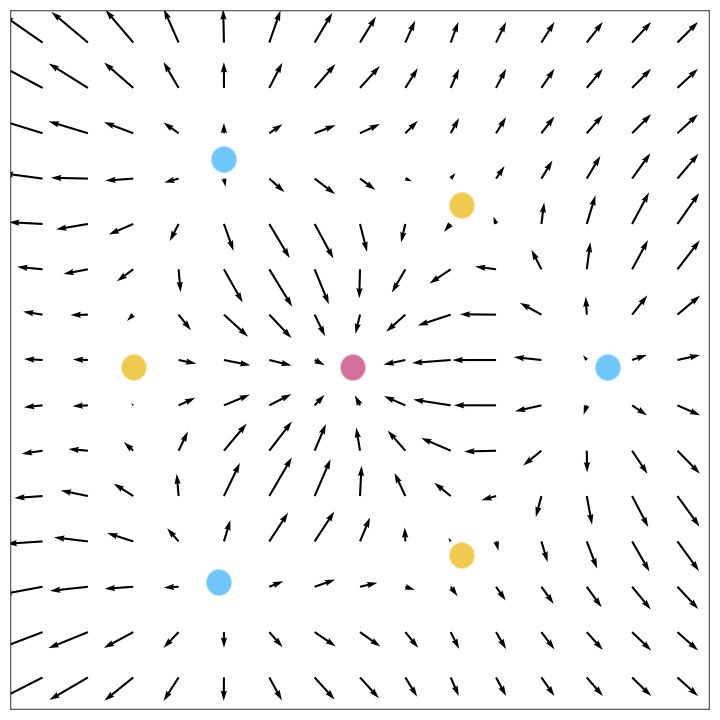}  
  \caption{\footnotesize After modification, six critical points are added into $X \backslash X^G$. $o$ has index $0$ in $X$ and index $0$ in $X^G$.}
\end{subfigure}
\caption{\footnotesize $G=C_3$ acts on $X=\R^2$ by rotation. $X^G = \{o\}$}
\label{fig: Z3 action}
\end{figure}
\vspace*{0cm}
\begin{figure}
\centering
\begin{subfigure}[t]{.45\textwidth}
  \centering
  \includegraphics[scale = 0.25]{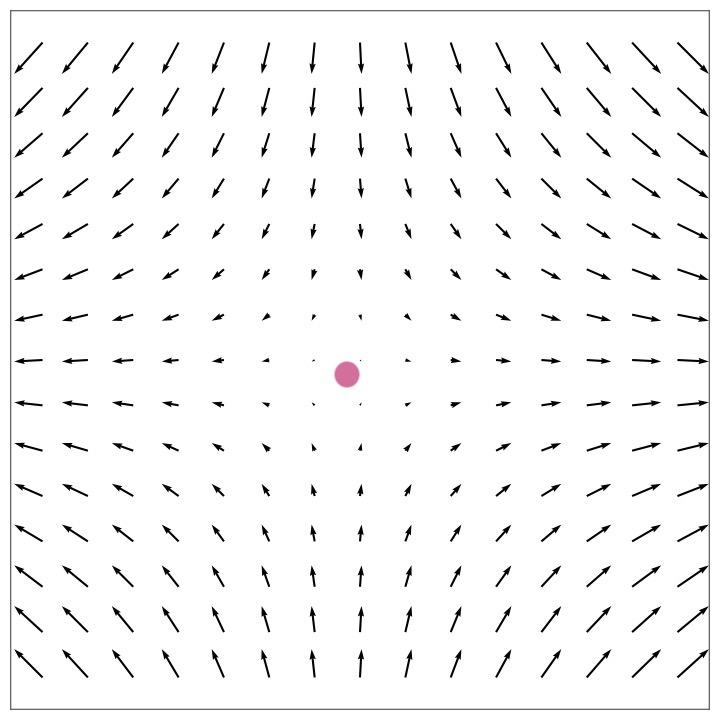} 
  \caption{Before modification, $o$ has index $1$ in $X$, and index $0$ in $X^G$.}
\end{subfigure}\quad\quad
\begin{subfigure}[t]{.45\textwidth}
  \centering
  \includegraphics[scale = 0.25]{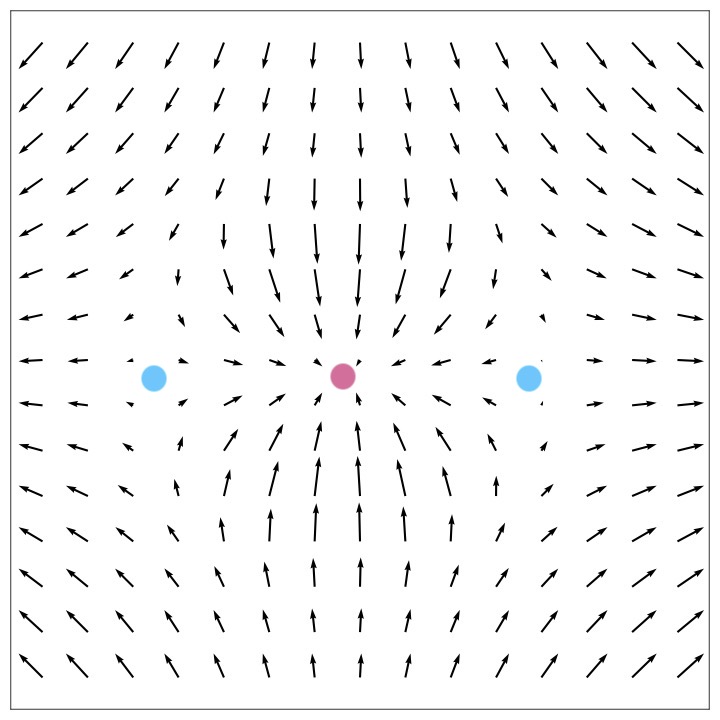}  
  \caption{\footnotesize After modification, two critical points are introduced into $X \backslash X^G$. $o$ has index $0$ in $X$ and index $0$ in $X^G$.}
\end{subfigure}
\caption{$G=C_2$ acts on $X=\R^2$ by reflection along the $y$-axis. $X^G = \{x = 0\}$.}
\label{fig: Z2 action}
\end{figure}

\section{Morse--Smale Metrics for Stable Morse Functions}
In this section, we demonstrate that for any stable Morse function, a generic equivariant metric ensures the Morse--Smale condition.
Smale's Theorem, originally due to Smale, establishes the genericity of the Morse--Smale condition.

\begin{theorem}[Smale's Theorem]\label{thm: smale theorem}
    Let $k \geq 1$ be an integer, and let $f$ be a Morse function of class $C^{k+1}$. Then for a generic Riemannian metric $\g$ of class $C^k$, the pair $(f, \g)$ is Morse--Smale.
\end{theorem}

The proof of this theorem for gradient-like vector fields can be found in \cite{smale1961ongradient}, and the proof for gradient vector fields is available in \cite{salamonMorse, hutchings2002lecture}.

Now, consider an equivariant Morse function $f$ on $M$ and an equivariant metric $g$ on $M$.
The goal of this section is to prove Theorem~\ref{thm: morse smale is generic fixing a stable f}, an equivariant version of Smale's theorem.
To this end, we start with the following lemmas:

\begin{lemma} \label{lemma: equivariant of tangent is the tangent of equivariant}
    For any $p \in M$, let $H \subset \stab(p)$ be a subgroup of the stabilizer of $p$.  
    Then the $H$-equivariant subspace $(T_p M)^H$ is identical to the tangent space $T_p (M^H)$.
\end{lemma}

\begin{proof}
    It is clear that $T_p (M^H) \subseteq (T_p M)^H$. For any $v \in (T_p M)^H$, we can extend $v$ to a vector field $V$ on $M$. Then, the vector field $W = \sum_{\sigma \in H} d\sigma (V)$ is an $H$-equivariant vector field on $M$. The flow of $W$ maps $M^H$ to itself. Let $\alpha: (-\epsilon, \epsilon) \to M$ be the integral curve of $W$ passing $p$, more precisely, $\alpha(0) = p$ and $\alpha' = W(\alpha)$. Then $\alpha$ is contained in $M^H$. This implies that $\alpha'(0) \in T_p (M^H).$ But on the other hand, $\alpha'(0) = v$.
\end{proof}

\begin{lemma}
    For any subgroup $H\subset G$,
    the set of critical points of the restricted function $f|_{M^H}$ is precisely the intersection of the critical point set $\crit(f)$ with the fixed-point submanifold $M^H$, i.e., $\crit(f|_{M^H}) = \crit(f) \cap M^H$.
\end{lemma}

\begin{proof}
    Suppose $p \in \crit(f|_{M^H})$ and assume to the contrary that there exists a non-zero vector $v \in (T_p M^H)^\perp \subset T_p M$ such that $df(v) > 0$, where $(T_p M^H)^\perp$ is the orthogonal complement of $T_p M^H$.
    Consider the vector $w = \sum_{\sigma \in H} d\sigma (v)$. Note that $df(w) = \sum_{\sigma \in H} df(d\sigma (v)) = \sum_{\sigma \in H} d(f \circ \sigma) (v) = \sum_{\sigma \in H} df  (v) > 0.$
    However, since $w$ is equivariant under $H$, we get $w\in T_p (M^H)$ by Lemma~\ref{lemma: equivariant of tangent is the tangent of equivariant}. But by assumption, $df = 0$ on $T_p M^H$, which contradicts the fact that $df(w) >0$.
\end{proof}

\begin{lemma}
    Suppose $p \in \crit(f)$ is a stable critical point and let $H = \stab(p)$ denote its stabilizer subgroup. Then the descending manifold $\unstableManifold{f}{g}{p}$ of $p$ is fixed by $H$, i.e., $\unstableManifold{f}{g}{p} \subset M^H$.
\end{lemma}

\begin{lemma}\label{lemma: transversality inside implies outside}
    Suppose $(f,g)$ is equivariant, and $f$ is stable. 
    Let $H$ be any subgroup of $G$.
    For $p,q \in \crit(f)\cap M^H$, let $\mathscr D = \unstableManifold{f}{g}{p}$ and $\mathscr A = \stableManifold{f}{g}{q}$. If $\mathscr D$ intersects $\mathscr A \cap M^H$ transversely inside $M^H$, then $\mathscr D$ intersects $\mathscr A$ transversely inside $M$.
\end{lemma}

\begin{proof}
    It suffices to show that $\mathscr A$ intersects $M^H$ transversely near $q$. Note that $T_q M = E_+ \oplus E_-$, where $E_\pm$ are the positive and negative eigenspaces of the Hessian $\hess{f}(q)$. Suppose $E_+ = E_+^H \oplus (E_+^H)^\perp$, where $E_+^H$ is its $H$-equivariant subspace, and $(E_+^H)^\perp$ is its orthogonal complement. By the stability assumption, $E_-$ is $H$-equivariant, hence $E_+^H \oplus E_- = T_q M^H$. Therefore, $T_q M = E_+ \oplus E_- = T_q \mathscr A + T_q M^H$, since $E_+^H \subset E_+ = T_q \mathscr A$. Hence, $\mathscr A$ and $M^H$ intersect transversely at $q$, so they intersect transversely in a small neighborhood $\openSet{U}_q$ in $M$ near $q$. Next, for any other point $x \in \mathscr A \cap M^H$, the negative gradient flow maps a sufficiently small neighborhood $\openSet{V}_x$ of $x$ diffeomorphically onto a small open subset $\openSet{W} \subset \openSet{U}_q$. The diffeomorphism also maps $M^H \cap \openSet{V}_x$ onto $M^H \cap \openSet{W}$, and maps $\mathscr A\cap \openSet{V}_x$ onto $\mathscr A\cap \openSet{W}$. But inside $\openSet{W}$, $M^H$ and $\mathscr A$ intersect transversely.
\end{proof}

\begin{corollary}\label{cor: key lemma for MS}
    Suppose $(f,g)$ is equivariant, and $f$ is stable.     
    For $p,q \in \crit(f)$, 
    let $H = \stab(p)$,
    $\mathscr D = \unstableManifold{f}{g}{p}$, and $\mathscr A = \stableManifold{f}{g}{q}$.
    Suppose that $\mathscr D$ intersects $\mathscr A \cap M^H$ transversely inside $M^H$. Then $\mathscr D$ intersects $\mathscr A$ transversely inside $M$.
\end{corollary}

\begin{proof}[Proof of Theorem~\ref{thm: morse smale is generic fixing a stable f}]
    We prove a stronger result:
    \begin{claim}
    Fix an equivariant, $C^k$-metric $g$. There exists an equivariant $C^k$-metric $g'$ such that
    \begin{enumerate}
    \item $g'$ is arbitrarily $C^k$-close to $g$,
    \item $g'$ and $g$ agree away from arbitrarily small neighborhoods of the critical points of $f$, and
    \item $(f,g')$ is Morse--Smale.
    \end{enumerate}
    \end{claim}
    The claim says that the Morse--Smale condition is dense in the $C^k$-topology of metrics.
    Together with the fact that the Morse-Smale condition is an open condition (see for example Proposition 3.4.3 in \cite{audin2010morse}), 
    we proved the theorem.
    
    The proof of this claim when $G = \{e\}$ follows from a straightforward modification of the proof of Theorem~\ref{thm: smale theorem}. Precisely, for any trajectory $\gamma$ of the negative gradient flow connecting two critical points, the transversality of the ascending and descending manifolds at $\gamma(0)$ is equivalent to the transversality at $\gamma(s)$ for any $s$. Therefore, by tracing back the proofs in \cite{salamonMorse, hutchings2002lecture}, one can see that it suffices to perturb the metric in small neighborhoods around the critical points. This proves the claim for the trivial group action case.

    For the general case, the argument is as follows: for any $p \in \crit(f)$, let $H = \stab(p)\subset G$ be its stabilizer subgroup. Let $V_p \subset M$ be a sufficiently small $H$-equivariant neighborhood of $p$ such that $V_{\sigma p} = \sigma(V_p)$ and $V_p \cap V_{\sigma p} = \emptyset$ for $\sigma \in G-H$. By the stability assumption, $\mathscr D_{g}(p) \subset M^H$. Note that for any other critical point $q \in \crit(f)$, if $\mathscr D_g(p) \cap \mathscr A_g(q) \neq \emptyset$, then $q \in M^H$. Moreover, $\mathscr A_g(q)\cap M^H = \mathscr A_{g|_{M^H}}(q)$, the ascending manifold of $q$ in $M^H$ with respect to the Morse function ${f|_{M^H}}$ and the metric $g|_{M^H}$. By the above claim for the trivial group action case, we perturb $g|_{M^H}$ in $M^H$ inside $V_p \cap M^H$ to a metric $g'_H$ of $M^H$ so that $\mathscr D_{g'_H}(p)$ intersects $\mathscr A_{g'_H}(q)$ transversely in $M^H$. Then we extend $g'_H|_{V_p \cap M^H}$ to an $H$-equivariant metric in $V_p$, denoted by $g'_{V_p}$, by taking an arbitrary extension and averaging over $H$. Finally, for any $\sigma\in G$, we define $g'_{V_{\sigma p}}$ by $\sigma_\sharp g'_{V_p}$, which is well-defined since $g'_{V_p}$ is $H$-equivariant. 
    By construction, $\mathscr D_{g'}(p)$ intersects $\mathscr A_{g'}(q) \cap M^H$ transversely in $M^H$ for any critical point $q$. By Corollary~\ref{cor: key lemma for MS}, $(f,g')$ is Morse-Smale.
\end{proof}

\section{Equivariant homology theories}

\subsection{Equivariant (co)homology and Bredon (co)homology}

In this section we will recall some details about equivariant homology theories. The interested reader should consult \cite[\S I.3]{May1996Alaskanotes} and the references there for more details.

\begin{definition}
For a fixed group $G$, an \emph{equivariant homology theory for $G$-spaces} is made up of functors
\[
(X,A) \mapsto h_n(X,A),
\]
assigning an abelian group to each pair of a $G$-space $X$ and a $G$-invariant subspace $A \subset X$. (By convention, we write $h_n(X)$ for $h_n(X,\emptyset)$.) These are required to satisfy analogues of the Eilenberg--Steenrod axioms.
\begin{itemize}
    \item Functoriality: it is functorial in $G$-equivariant maps $(X,A) \to (Y,B)$.
    \item Homotopy invariance: any two functions that are $G$-equivariantly homotopic are taken to the same map by $h$.
    \item Additivity: disjoint unions are taken to direct sums.
    \item Excision: for a $G$-invariant subspace $V$ such that $\overline V \subset A^\circ$, the map $(X\setminus V, A \setminus V) \to (X,A)$ induces an isomorphism on $h_*$.
    \item Long exact sequence: there are natural transformations $\partial: h_n(X,A) \to h_{n-1}(A)$ such that the sequence
    \[
    \cdots \to h_n(A) \to h_n(X) \to h_n(X,A) \to h_{n-1}(A) \to \cdots
    \]
    is exact. 
\end{itemize}
This homology theory is \emph{ordinary} if it satisfies an additional axiom.
\begin{itemize}
    \item Dimension: for any subgroup $H \leq G$, $h_n(G/H) = 0$ for $n \neq 0$.
\end{itemize}
Similarly, we define equivariant cohomology theories.
\end{definition}

\begin{example}
    For any subgroup $H \leq G$ and any $K \leq W_G H$, the functor
    \[
        (X,A) \mapsto
        H_n((X/H)^K, (A/H)^K)
    \]
    is an ordinary equivariant homology theory.
\end{example}

\begin{example}
    Borel equivariant homology
    \[
        (X,A) \mapsto H_n^G(X,A) = H_n(EG \times_G X, EG \times_G A)
    \]
    is a (non-ordinary) equivariant homology theory.
\end{example}

\begin{definition}
    The \emph{orbit category} $\orb_G$ is the category of $G$-sets of the form $G/H$ for $H$ any subgroup of $G$. In particular, for any subgroups $H$ and $K$, maps $G/H \to G/K$ are of the form $gH \mapsto gxK$ for $xK \in (G/K)^H$.
    
    A \emph{covariant Bredon coefficient system} is a functor
    \[
    \underbar M\co \orb_G \to \mathrm{Ab},
    \]
    and a \emph{contravariant Bredon coefficient system} is a functor
    \[
    \underbar N\co \orb_G^{op} \to \mathrm{Ab}.
    \]
\end{definition}

\begin{remark}
    When the group $G$ is fixed, for convenience it is common to write $\underbar M(H)$ rather than $\underbar M(G/H)$.
\end{remark}

In particular, any equivariant homology theory $h_*$ has associated covariant Bredon coefficient systems $\underbar h_n\co G/H \mapsto h_n(H)$, and any equivariant cohomology theory $h^*$ has an associated contravariant Bredon coefficient systems $G/H \mapsto h^n(H)$. The following is a converse.

\begin{theorem}
    Ordinary homology theories are determined by their underlying coefficient system, as follows.
    \begin{enumerate}
        \item Associated to any covariant coefficient system $\underbar M$, there is an equivariant homology theory $H^{Br}_*(-;\underbar M)$, called \emph{Bredon homology with coefficients in $\underbar M$}.
        \item Bredon homology $H^{Br}_*(-;\underbar M)$ is functorial in $\underbar M$.
        \item Bredon homology takes \emph{weak} equivalences of $G$-spaces to isomorphisms.
        \item If $h_*$ is an ordinary homology theory with underlying coefficient system $\underbar h_0$, then there is a natural transformation $H^{Br}_*(-;\underbar h_0) \to h_*(-)$ which is an isomorphism for all $G$-CW pairs $(X,A)$. 
    \end{enumerate}
    A dual result holds for ordinary cohomology theories.
\end{theorem}

\begin{remark}
    Recall that the pair $(X,A)$ is a $G$-CW pair, if $X$ is a $G$-CW complex, and $A$ is a $G$-CW subcomplex.
\end{remark}

\begin{remark}
    Bredon (co)homology can be constructed in a manner similar to singular homology: for a $G$-space $X$, we have a $G$-set $S_n(X)$ of continuous maps $\sigma\co \Delta^n \to X$, and we define
    \[
    C^{Br}_n(X;\underbar M) = \bigoplus_{[\sigma] \in S_n(X) / G} \underbar M(\stab(\sigma)).
    \]
    This has an appropriate alternating-sign simplicial boundary map, using the fact that the stabilizer of a singular simplex is contained in the stabilizer of any face.
\end{remark}

\begin{remark}
    Even though Borel equivariant (co)homology is not ordinary, it can be defined using hyper(co)homology. Given a chain complex $\underbar M_*$ of Bredon coefficient systems, there is an associated double complex
    \[
    C_n^{Br}(X; \underbar M_*) = \bigoplus_{[\sigma] \in S_n(X) / G} \underbar M_*(\stab(\sigma))
    \]
    producing an equivariant homology theory on $G$-spaces. Similar results apply for a cochain complex $\underbar N^*$.
\end{remark}

\subsection{Filtrations and spectral sequences}

Suppose $X$ is a $G$-space with a filtration
\[
\emptyset \subset F^0 X \subset F^1 X \subset \dots \subset X
\]
by $G$-invariant subspaces. Then, for any equivariant homology theory $h_*$, we get interconnected long exact sequences:
\[
\dots \to h_*(F^{p-1} X) \to h_*(F^p X) \to h_*(F^p X, F^{p-1} X) \to h_{*-1} (F^{p-1} X) \to \dots
\]
Similarly, there are long exact sequences for an equivariant cohomology theory $h^*$. These ``exact couples'' immediately give rise to the following result \cite{Boardman1999Conditional}.
\begin{proposition}
\label{prop:filtrationsseq}
    Suppose that $X$ has a filtration by subspaces $F^p X$. If $h_*$ is a homology theory such that
    \[
    \varinjlim_p h_* F^p X \to h_*(X)
    \]
    is an isomorphism, then there is a strongly convergent, homologically graded spectral sequence with $E^1$-term
    \[
    E^1_{p,q} = h_{p+q}(F^p X, F^{p-1} X) \Rightarrow h_{p+q} X.
    \]
    Similarly, if $h^*$ is a cohomology theory, then there is a conditionally convergent, cohomologically graded spectral sequence with $E_1$-term
    \[
    E_1^{p,q} = h^{p+q}(F^p X, F^{p-1} X) \Rightarrow h^{p+q} X.
    \]
\end{proposition}
\begin{remark}
    All convergence is automatic if the filtration is finite. 
\end{remark}

\begin{example}
\label{ex:homological}
    If $X$ is a $G$-CW complex, then we can apply this result to the cellular filtration $\emptyset \subset X^{(0)} \subset X^{(1)} \subset \dots$. By standard excision techniques, we get an isomorphism:
    \begin{align*}
    h_{p+q}(X^{(p)}, X^{(p-1)})
    &\cong \bigoplus_{[\alpha]} h_{p+q}(G/\stab(\alpha) \times D^p, G/\stab(\alpha) \times S^{p-1})\\
    &\cong \bigoplus_{[\alpha]} \widetilde h_{p+q}(\Sigma^p G/\stab(\alpha))\\
    &\cong \bigoplus_{[\alpha]} \underbar h_q(\stab(\alpha))
    \end{align*}
    Here the sums are over orbits of $p$-cells, and $\underbar h_q$ are the coefficient systems associated to $h_*$. The result is an Atiyah--Hirzebruch style spectral sequence:
    \[
    H_p^{Br}(X; \underbar h_q) \Rightarrow h_{p+q}(X).
    \]
    Similar results hold for cohomology.
\end{example}

\subsection{Equivariant Morse homology}

The following tool is handy.
\begin{lemma}[2-out-of-6]
    Suppose that we have maps $X \xrightarrow{f} Y \xrightarrow{g} Z \xrightarrow{h} W$ in a category $C$ such that the composites $gf\co X \to Z$ and $hg\co Y \to W$ are isomorphisms. Then $f$, $g$, and $h$ are isomorphisms.
\end{lemma}

\begin{proof}
    The map $g$ has a left inverse $(hg)^{-1} h$ and a right inverse $f(fg)^{-1}$, and so it is invertible.
\end{proof}

In particular, taking $C$ to be the the homotopy category of $G$-spaces, if $f$, $g$, and $h$ are maps of $G$-spaces such that $gf$ and $hg$ are equivariant homotopy equivalences, then $f$, $g$, and $h$ are equivariant homotopy equivalences.

\begin{lemma}
\label{lemma:flowlemma}
    Suppose $M$ is a closed $G$-manifold with an equivariant Morse function $f$ and a $G$-invariant Riemannian metric $g$, with associated gradient flow $(p,t) \mapsto \phi_t(p)$.

    Suppose $Z \subset Y \subset M$ is an inclusion of two such subsets such that $f$ has no critical points in $\overline Y \setminus Z^\circ$, and such that $Y \setminus Z \subset f^{-1}([a,b])$. Then the inclusion $Z \to Y$ is an equivariant homotopy equivalence.
\end{lemma}

\begin{proof}
    Invariance of the Morse function $f$ and of the metric $g$ implies that $\phi_t$ is equivariant. Note that the gradient flow satisfies $\tfrac{\partial}{\partial t}(f(\phi_t(x)) = ||\nabla f(\phi_t x)||^2$. 
 
    Suppose that $Y \subset M$ is any subset such that $\phi_t Y \subset Y$ for $t \geq 0$. Then the inclusion $\phi_t Y \to Y$ is then an equivariant deformation retract: the equivariant homotopy is $H(s,x) = \phi_{ts} x$.
    
    Suppose that $Z \subset Y \subset M$ is an inclusion of two such subsets such that $\phi_t Y \subset Z$ for sufficiently large $t$. We can then apply the 2-out-of-6 lemma to the composite
    \[
        \phi_t Z \subset \phi_t Y \subset Z \subset Y.
    \]
    We find that the inclusion $Z \to Y$ is an equivariant homotopy equivalence.

    Suppose $Z \subset Y \subset M$ is an inclusion of two such subsets such that $f$ has no critical points in $\overline Y \setminus Z^\circ$. Then on the closed set $\overline Y \setminus Z^\circ$, $||\nabla f(x)||^2$ is positive, and so by compactness it has a lower bound $\delta > 0$. Then for all $x \in Y$ and all $t > 0$, either $\phi_t x \in Z$ or $f(\phi_t(x)) \leq f(x) - t \delta$ for all $t \geq 0$.
    
    In particular, if $Y \setminus Z \subset f^{-1}([a,b])$, then $f^{-1}(-\infty,a] \subset Z \subset Y \subset f^{-1}(-\infty,b]$, and so then this forces $\phi_{(b-a)/\delta} Y \subset Z$.
\end{proof}

\begin{proposition}
    Suppose that $M$ is a closed $G$-manifold equipped with an equivariant Morse function $f$.
    
    \begin{enumerate}
        \item For any $a < b$ such that $f$ has no critical values in $[a,b]$, the inclusion $f^{-1}(-\infty,a] \to f^{-1}(-\infty,b]$ is an equivariant homotopy equivalence.
        \item Suppose that $a < c$ and $c$ is the only critical value of $f$ in $[a,c]$, with associated critical points $p_1,\dots,p_n$. Then there exist arbitrarily small open coordinate balls $B_i$ around $p_i$, there is an inclusion
        \[
        f^{-1}(-\infty,a] \to f^{-1}(-\infty,c] \setminus (\cup_{i=1}^n B_i)
        \]
        which is an equivariant homotopy equivalence.
        \item Suppose that $c < b$ and $c$ is the only critical value of $f$ in $[c,b]$, with associated critical points $p_1,\dots,p_n$. Then there exist arbitrarily small closed coordinate balls $\bar B_i$ around $p_i$, there is an inclusion
        \[
        f^{-1}(-\infty,c] \cup (\cup_{i=1}^n \overline B_i) \to f^{-1}(-\infty,b]
        \]
        which is an equivariant homotopy equivalence.
        \item Suppose that $a < c < b$ and $c$ is the only critical value of $f$ in $[a,b]$, with associated critical points $p_1,\dots,p_n$. Then passage across the critical value is homotopy equivalent to equivariant cell attachment: there is an equivariant homotopy equivalence
        \[
        f^{-1}(\infty,b] \simeq f^{-1}(-\infty,a] \cup_{\cup S(T_{p_i} \mathscr D)} \bigcup D(T_{p_i} \mathscr D)
        \]
    \end{enumerate}
\end{proposition}

In short: an equivariant Morse function on $M$ gives us a description (up to homotopy equivalence) of $M$ via iterated \emph{representation cell attachment}. For each critical point $p$ of $M$, the descending tangent space $T_p \mathscr D$ is a representation of $\stab(p)$, and $M$ is formed by iteratively gluing together equivariant cells associated to the unit discs in these representations.

\begin{proof}
    These results are all applications of Lemma~\ref{lemma:flowlemma}.
    \begin{enumerate}
        \item Choosing any equivariant metric $g$, this follows by applying the lemma to the inclusion
        \[
        f^{-1} (-\infty,a] \subset f^{-1} (-\infty,b].
        \]
        \item In a neighborhood of each point $p_i$, choose an equivariant Euclidean coordinate chart $U_i \cong V_i \oplus W_i$ around $p_i$ such that in these coordinates, $f(v,w) = ||v||^2 - ||w||^2$. Let $g$ be the standard Euclidean metric near the points $p_i$, extended to a metric on all of $M$.
        
        Choose $B_i$ to be a sufficiently small $\epsilon$-ball around $p_i$ in this metric so that $f(x) > a$ on $B_i$. The gradient flow of $f$ is the standard Euclidean gradient flow on $B_i$. The result follows by applying the lemma to the inclusion
        \[
        f^{-1}(-\infty,c] \setminus (\cup_{i=1}^n B_i) \subset f^{-1}(-\infty,a]
        \]
        because the first space is closed under $\phi_t$.
        \item This follows from choosing local Euclidean metrics just as in the previous case.
        \item Choosing an equivariant metric as in the previous two cases: this has balls $B_i$ on which the equivariant Morse function is of the form $f(v,w) = ||v||^2 - ||w||^2$, and let $D_i \subset B_i$ be the descending part of points of the form $(0,w)$ in $B_i$. There is a diagram
        \[
        \adjustbox{scale=0.75}{%
        \begin{tikzcd}
            f^{-1}(-\infty,a] \ar[d,hookrightarrow] &
            f^{-1}(-\infty,a] \ar[l, equals] \ar[r] \ar[d,hookrightarrow] &
            f^{-1}(-\infty,c] \setminus (\cup B_i) \ar[d,hookrightarrow]  &
            f^{-1}(-\infty,c] \setminus (\cup B_i) \ar[d,hookrightarrow]  \ar[l,equals]
            \\
            f^{-1}(-\infty,b] &
            f^{-1}(-\infty,c] \cup (\cup \overline B_i) \ar[r, equals] \ar[l] &
            f^{-1}(-\infty,c] \cup (\cup \overline B_i) &
            f^{-1}(-\infty,c] \setminus (\cup (B_i \setminus D_i)) \ar[l,hookrightarrow]        \end{tikzcd}        
        }
        \]
        where the horizontal maps are equivariant homotopy equivalences. The inclusion $f^{-1}(-\infty,a] \hookrightarrow f^{-1}(-\infty,b]$ is therefore equivariantly homotopy equivalent to the right-hand inclusion, which is the desired equivariant cell attachment.
    \end{enumerate}
\end{proof}

This iterative cell attachment gives us a tool to calculate equivariant homology, based on the homology of representation cells.

\begin{definition}
    Suppose that $h_*$ is an equivariant homology theory. For any subgroup $H \leq G$ and any Euclidean $H$-representation $V$,  with unit disc $D(V)$ and unit sphere $S(V)$, we define the \emph{representation cell groups}
    \[
    h^H_n(V) = h_n(G \times_H D(V), G \times_H S(V))
    \]
    to be the homology of the associated free orbit of representation cells. Similarly, for an equivariant cohomology theory $h^*$, we get representation cell groups
    \[
    h^n_H(V) = h^n(G \times_H D(V), G \times_H S(V)).
    \]
\end{definition}

\begin{theorem}
    Suppose that $M$ is a closed $G$-manifold equipped with an equivariant Morse function $f$. Let $c_1 < c_2 < \dots < c_k$ be the critical values, let $p_{i,j}$ be chosen representatives for the critical orbits with critical values $c_i$, and $V_{i,j}$ the $\stab(p_{i,j})$-representation for the descending tangent space. Then for any equivariant homology theory $h_*$, there is a spectral sequence with $E_1$-term
    \[
    E^1_{n,m} = \bigoplus_{j} h^{\stab(p_{i,j})}_{n+m}(V_{n,j}) \Rightarrow h_{n+m}(M).
    \]
    Similarly, for any equivariant cohomology theory $h^*$, there is a spectral sequence
    \[
    E_1^{n,m} = \bigoplus_{j} h^{n+m}_{\stab(p_{i,j})} (V_{n,j}) \Rightarrow h^{n+m}(M).
    \]    
\end{theorem}

\begin{example}
    Suppose that we have a manifold with boundary $N$ and we form the associated double $M = N \cup_{\partial N} N$, with action of the cyclic group $G = C_2$ by reflection across the boundary; the quotient $M/G$ is isomorphic to $N$. Choose a Morse function on $N$ transverse to the boundary, which determines a Morse function $f$ on $M$.
    
    Critical points of this Morse function come in three types.
    \begin{itemize}
        \item \emph{Interior} critical points from $N \setminus \partial N$, which come in $G$-equivalent pairs. The stabilizer is the trivial group, and the associated representation is $\R^{\morseIndex{f}{p}}$.
        \item \emph{Stable} critical points $p$ in $\partial N$, whose descending submanifold is contained in $N$. The stabilizer is $G$, and the associated representation of $G$ is a trivial representation $\R^{\morseIndex{f}{p}}$.
        \item \emph{Unstable} critical points in $\partial N$ whose descending submanifold is not contained in $N$. The stabilizer is $G$, and the associated representation of $G$ is the sum of a trivial representation $\R^{(\morseIndex{f}{p})-1}$ and the nontrivial one-dimensional representation, denoted $\R^-$.
    \end{itemize}
    Different equivariant homology theories give different values on these representation cells, and so they can appear in different ways on the $E_1$-page.
    \begin{itemize}
        \item The homology theory $X \mapsto H_*(X)$ sends a critical point of index $k$ to $\Z[G]$ in degree $k$ if it is interior, $\Z$ in degree $k$ if it is stable, and $\Z$ in degree $k$ if it is unstable.
        \item The homology theory $X \mapsto H_*(X^G)$ sends a critical point of index $k$ to $0$ if it is interior, $\Z$ in degree $k$ if it is stable, and $\Z$ in degree $(k-1)$ if it is unstable.
        \item The homology theory $X \mapsto H_*(X/G)$ sends a critical point of index $k$ to $\Z$ in degree $k$ if it is interior, $\Z$ in degree $k$ if it is stable, and $0$ if it is unstable. In particular, \emph{interior} and \emph{stable} critical points are the only contributors to $H_*(M/G) = H_*(N)$.
        \item The homology theory $X \mapsto H_*(X/G, X^G)$ sends a critical point of index $k$ to $\Z$ in degree $k$ if it is interior, $0$ if it is stable, and $\Z$ in degree $k$ if it is unstable. In particular, \emph{interior} and \emph{unstable} critical points are the only contributors to $H_*(M/G, M^G) = H_*(N, \partial N)$.
    \end{itemize}
    For these last two, the reader should compare \cite[Theorem 2.4.5]{Kronheimer2007Monopoles}.
\end{example}

\section{Equivariant Thom-Smale-Witten complex}

Consider a manifold $M$ equipped with an equivariant, stable Morse function $f$, and let $g$ be an equivariant Riemannian metric such that the pair $(f, g)$ is Morse--Smale.


\begin{definition}
For any $k$, we define $F^k M \subset M$ to be the union of descending manifolds of critical index at most $k$:
\[
F^k M = \bigcup_{\morseIndex{f}(p) \leq k} \D{p}.
\]    
\end{definition}




By \cite[Theorem B]{qin2021morseCW}, which requires the pair $(f,g)$ being Morse-Smale, the spaces $F^k M$ are the skeleta in a CW-decomposition of the manifold $M$. Moreover, stability implies that the descending tangent space at $p$ is isomorphic to the trivial representation of $\stab(p)$. This makes $F^k M$ into a $G$-CW decomposition of $M$, and we get an associated spectral sequence from Example~\ref{ex:homological}.
\begin{proposition}
Suppose that $h_*$ is an equivariant homology theory. Then there is a strongly convergent, homologically graded spectral sequence with $E^1$-term
\[
E^1_{n,m} = \bigoplus_{[p]} \underbar h_m(\stab(\alpha)) \Rightarrow h_{n+m}(M)
\] 
Here the sum is over $G$-orbits of critical points $p$ of index $n$. The $E^2$-page is made up of the Bredon homology groups
\[
E^2_{n,m} = H_n^{Br}(M; \underbar h_m).
\]

Similarly, for an equivariant cohomology theory $h^*$, there is a strongly convergent, cohomologically graded spectral sequence with $E_1$-term
\[
E_1^{n,m} = \bigoplus_{[p]} \underbar h^m(\stab(\alpha)) \Rightarrow h^{n+m}(M)
\]
The $E_2$-page is made up of the Bredon cohomology groups
\[
E_2^{n,m} = H^n_{Br}(M; \underbar h^m).
\]
\end{proposition}

\begin{remark}[Equivariant Thom-Smale-Witten complex -- a canonical approach]
    In the case where we take Bredon homology with coefficients in a covariant coefficient system $\underbar M$, then the terms in this spectral sequence are trivial for $m \neq 0$. In this case, we are left with a \emph{Thom-Smale-Witten complex}
    \[
    \dots 
    \to \bigoplus_{[p], \morseIndex{f}(p) = 1} \underbar M(\stab(p)) 
    \to \bigoplus_{[p], \morseIndex{f}(p) = 0} \underbar M(\stab(p)) \to 0
    \]
    which computes the ordinary Bredon homology groups $H_*^{Br}(X;\underbar M)$. From this point of view, we do not have to choose orientations and differentials to define a Thom-Smale-Witten complex or prove that it is invariant. 
\end{remark} 
\begin{remark}[Equivariant Thom-Smale-Witten complex -- an explicit approach]
    We review the standard equivariant Thom-Smale-Witten complex as follows:  
    for each critical point $p \in \crit(f)$, we fix an orientation of the descending manifold $\mathscr D_p$. Define the chain complex  
    \[ C_* = \bigoplus_{p \in \crit(f)} \mathbb{Z} \langle [p] \rangle. \]  
    The differential $\p$ is given on generators by  
    \[
    \p [p] = \sum_{\morseIndex{f}(q) = \morseIndex{f}(p) - 1} \sharp (\mathcal M(p,q)/\mathbb{R}) [q],
    \]
    where:
    \begin{enumerate}
        \item $\mathcal M(p,q)$ denotes the space of trajectories $\gamma: \mathbb{R} \to M$ satisfying  
         \[
         \dot \gamma = -\op{grad}_f(\gamma),
         \]
         with $\lim_{s\to -\infty}\gamma(s) = p$ and $\lim_{s \to \infty} \gamma(s) = q$;
         \item $t \in \R$ acts on $\mathcal M(p,q)$ by translation: $(\op{tr}_t \gamma) (s) = \gamma(s + t)$;
        \item By the Morse-Smale condition, $\mathcal M(p,q)/\mathbb{R}$ is a finite set;
        \item The space $\mathcal M(p,q)/\mathbb{R}$ is oriented as usual by comparing the orientations of the descending manifolds $\mathscr D_p$ and $\mathscr D_q$ (see, for example, \cite{salamonMorse});
        \item The signed count of points in $\mathcal M(p,q)/\mathbb{R}$ is denoted by $\sharp (\mathcal M(p,q)/\mathbb{R})$.
    \end{enumerate}  
    For any $\varphi \in G$, the differential $d\varphi$ of $\varphi$ induces a map $d\varphi|_p: T_p \mathscr D_p \to T_{\varphi p} \mathscr D_{gp}$. We define  
    \[
    \op{sign}(g,p) =  
    \begin{cases}  
    1, & \text{if } dg|_p \text{ preserves orientation}, \\  
    -1, & \text{if } dg|_p \text{ reverses orientation}.  
    \end{cases}  
    \]
    The group $G$ acts on $C_*$ via  
    \[
    g[p] = \op{sign}(g,p)[gp].
    \]  
    It is straightforward to check that this $G$-action commutes with $\p$ (see \cite{bao2021equivariant}).  
    \end{remark}

\begin{remark}
For a finite $p$-group $G$ and a $G$-space $X$ whose fixed-point spaces $X^H$ have finite-dimensional mod-$p$ homology, calculations in Bredon homology are enough to deduce the \emph{Smith inequalities} proved in \cite{floyd1952on}. For all $\ell$,
\[
\sum_{k \geq \ell} \dim H_k(X; \mathbb F_p)\geq \sum_{k \geq \ell} \dim H_k(X^G; \mathbb F_p).
\]
A proof of this is given by May in \cite{may1987smith}, and Putman elaborates on it in \cite{putman2018smith}.

Alternatively, we show a special case of the Smith inequalities using the equivariant Thom-Smale-Witten complex $(C_*, \p)$. Suppose $G$ is the cyclic group of order $p$, written as $G = \{1, \varphi, \varphi^2, \dots, \varphi^{p-1}\}$. Tensoring $(C_*, \p)$ with $\mathbb F_p$ over $\Z$, we denote the resulting chain complex still by $(C_*, \p)$.
The $G$-action means that $(C_*, \p)$ is an $\mathbb F_p[G]$-module. Consider the Tate double complex $(E_{*,*}, d) = \oplus_{i,j} E_{i,j}$ defined by $E_{i,j} = C_j$. The differential $d = d_\vee + d_<$, where 
\be 
\item $d_\vee: E_{i,j} \to E_{i, j-1}$ is the map $$c_j \mapsto (-1)^i \p c_j$$ for any $c_j \in C_j$; 
\item $d_<: E_{i,j} \to E_{i-1, j}$ is the map 
    \begin{equation*}
    c_j \mapsto 
    \begin{cases}
        (1-\varphi)c_j & \text{ if } i \text{ is odd}, \\
        (1+ \varphi + \dots + \varphi^{p-1})c_j & \text{ if } i \text{ is even,} 
    \end{cases}
    \end{equation*}
    for any $c_j \in C_j$.
\ee 
The vertical first spectral sequence of $(E_{*,*}, d)$ has $E_1$-term ${}^\vee E^1_{i,j} = H_j(X;\mathbb F_p)$ and it converges to the homology of $(E_{*,*}, d)$. This implies $\sum_k \dim H_k(X; \mathbb F_p) \geq \dim H_0(E_{*,*},d)$
The horizontal first spectral sequence of $(E_{*,*}, d)$ has the $E_2$-term ${}^<E^2_{i,j} = H_j(X^G; \mathbb F_p)$, converges to the homology of $(E_{*,*}, d)$, and collapses at the $E_2$-page. This implies $\sum_k \dim H_k(X^G; \mathbb F_p) = \dim H_0(E_{*,*},d)$. Putting these two equations together, we obtain:
\[
    \sum_k \dim H_k(X; \mathbb F_p) \geq \sum_k \dim H_k(X^G; \mathbb F_p).
\]
\end{remark}

\printbibliography
\end{document}